\documentclass[12pt, a4paper, fceqn]{article}
\usepackage{lipsum}
\usepackage{authblk}
\usepackage{amsfonts}
\usepackage{amssymb}
\usepackage{amscd}
\usepackage{mathrsfs}
\usepackage{amsmath,amssymb, amsthm}
\usepackage{graphicx}
\usepackage{color}
\usepackage{authblk}
\usepackage{amssymb}
\usepackage{indentfirst,latexsym}
\usepackage[bookmarks=false]{hyperref}
\usepackage{amsmath}

\newtheorem{theorem}{Theorem}[section]

\newtheorem{conj}[theorem]{Conjecture}
\newtheorem{lem}[theorem]{Lemma}

\theoremstyle{definition}

\def\Sym{\hbox{\rm Sym}}

\usepackage{enumerate}
\long\def\delete#1{}

\usepackage{xcolor}
\usepackage[normalem]{ulem}

\title{The last patch for classifying shuffle groups}

\author{Junyang Zhang\footnote{Supported by the Natural Science Foundation of Chongqing (CSTB2022NSCQ-MSX1054). E-mail: jyzhang@cqnu.edu.cn (Junyang Zhang)}}
\affil{{\small School of Mathematical Sciences, Chongqing Normal University,\\ Chongqing 401331, P. R. China}}
\date{}

\begin{document}

\openup 0.5\jot
\maketitle

\vspace{-10mm}
\begin{abstract}

Divide a deck of $kn$ cards into $k$ equal piles and place them from left to right. The standard shuffle $\sigma$ is performed by picking up the top cards one by one from left to right and repeating until all cards have been picked up. For every permutation $\tau$ of the $k$ piles, use $\rho_{\tau}$ to denote the induced permutation on the $kn$ cards. The shuffle group $G_{k,kn}$ is generated by $\sigma$ and the $k!$ permutations $\rho_{\tau}$.
It was conjectured by Cohen et al in 2005 that the shuffle group $G_{k,kn}$ contains $A_{kn}$ if $k\geq3$, $(k,n)\ne\{4,2^f\}$ for any positive integer $f$ and $n$ is not a power of $k$. Very recently, Xia, Zhang and Zhu reduced the proof of the conjecture to that of the $2$-transitivity of the shuffle group and then proved the conjecture under the condition that $k\ge4$ or $k\nmid n$. In this paper, we proved that the group $G_{3,3n}$ is $2$-transitive for any positive integer $n$ which is a multiple of $3$ but not a power of $3$. This result leads to the complete classification of the shuffle groups $G_{k,kn}$ for all $k\ge2$ and $n\ge1$.


\end{abstract}

\section{Introduction}
Usually, to perfectly shuffle a deck of $2n$ cards, the cards is divided exactly in half and the two halves are shuffled so that they are perfectly interlaced.
There are two kinds of such shuffle which are referred to as the out-shuffle and in-shuffle. The out-shuffle leaves the original top card of the deck on top of the shuffled deck, while the in-shuffle moves the original top card to the second from the top of the shuffled deck. The permutation group generated by these two shuffles was completely determined for all $n$ by Diaconis, Graham, and Kantor \cite{DGK1983} in 1983. At the end of \cite{DGK1983}, it was suggested to study the so-called ``many handed shuffler": Divide a deck of $kn$ cards into $k$ equal piles and place them from left to right. Then there are $k!$ perfect shuffles as there are $k!$ possible orders of picking up the piles
to perfectly interleave. The standard shuffle $\sigma$ is performed by picking up the top cards one by one from left to right and repeating until all cards have been picked up. For a permutation $\tau$ of the $k$ piles, we have an induced permutation $\rho_{\tau}$ on the $kn$ cards which keeps the order of the cards within each pile. Note that for every $\tau$, the
composition $\rho_{\tau}\sigma$ (performed from left to right) is a perfect shuffle. In this way, we obtain all the $k!$ perfect shuffles.

\smallskip
Let $[m]:=\{0,1,\ldots,m-1\}$ for every positive integer $m$. For a deck of $kn$ cards being evenly divided into $k$ piles, label the piles from left to right as $0,1,\ldots,k-1$. Denote the set of the positions of the $kn$ cards by $[kn]$ such that for a card in position $i+jn$, we mean it is the $i$-th card from top to bottom in the pile $j$ with the top one in pile $j$ being the $(0+jn)$-th card. Then every perfect shuffle happens to be a permutation on $[kn]$. In particular, the standard shuffle $\sigma$ is defined by
\begin{equation*}
  (i+jn)^{\sigma}=ki+j~\mbox{for all}~i\in[n]~\mbox{and}~j\in[k].
\end{equation*}
For a permutation $\tau$ on $[k]$, the induced permutation $\rho_{\tau}$ on $[kn]$ is defined by
\begin{equation*}
  (i+jn)^{\rho_{\tau}}=i+j^{\tau}n~\mbox{for all}~i\in[n]~\mbox{and}~j\in[k].
\end{equation*}
Then $\{\rho_{\tau}\sigma\mid\tau\in \Sym([k])\}$ is the set all perfect shuffles. We call the subgroup $G_{k,kn}$ of $\Sym([kn])$ generated by $\{\rho_{\tau}\sigma\mid\tau\in \Sym([k])\}$ the shuffle group on $kn$ cards. It is obvious that $G_{k,kn}=\langle\sigma,\rho_{\tau}\mid\tau\in \Sym([k])\rangle$.

\smallskip
Recall that Diaconis, Graham and Kantor \cite{DGK1983} determined $G_{2,2n}$ for every positive integer $n$ and suggested to investigate $G_{k,kn}$ for general $k$. In 1987, Medvedoff and Morrison \cite{MM1987} proved that for every positive integer $e$ the group $G_{k,k^{e+1}}$ is the primitive wreath product
of $\Sym([k])$ by the cyclic group of order $e+1$. It is also conjectured in \cite{MM1987} that the group $G_{4,2^{e}}$ with $e$ being odd is isomorphic to the group of affine transformations of an $e$-dimensional vector space over the field of order $2$. This conjecture was proved in 2005 by Cohen, Harmse, Morrison and Wright \cite{CHMW}. In order to give a direction for a complete classification of the shuffle groups $G_{k,kn}$ for all $k\ge2$, the following  conjecture was posed in \cite{CHMW}.
\begin{conj}
\label{shuffle}
The shuffle group $G_{k,kn}$ contains $A_{kn}$ if $k\geq3$, $(k,n)\ne\{4,2^f\}$ for any positive integer $f$ and $n$ is not a power of $k$.
\end{conj}

In 2021, among other results, Amarra,  Morgan and  Praeger confirmed this conjecture for three  infinite families of pairs of integers: all $(k, n)$
with $k > n$; all $(k, n)\in\{(\ell^e,\ell^f)\mid \ell\ge2, \ell^e>4, e\nmid f\}$; and all $(k, n)$ with $k=2^e\ge4$ and $n$ not a power of $2$. Very recently, Xia, Zhang and Zhu \cite{XZZ} proved that Conjecture \ref{shuffle} is true if the shuffle group $G_{k,kn}$ is $2$-transitive and confirmed the $2$-transitivity of $G_{k,kn}$ under the condition that $k\ge4$ or $k\nmid n$. They also conjectured that $G_{3,3n}$ is $2$-transitive for any positive integer $n$ which is a multiple of $3$ but not a power of $3$. By analyzing the parity of generators of $G_{k,kn}$, it is easy to check that  $G_{k,kn}$ is contained in $A_{kn}$ if and only if either $n\equiv2\pmod4$ and $k\equiv0~\mbox{or}~1\pmod4$, or $n\equiv0 \pmod4$ (see, for example, \cite[Theorem 1]{MM1987}). Therefore, the last patch for classifying all shuffle groups is to prove the $2$-transitivity of $G_{3,3n}$ where $n$ is a multiple of $3$ but not a power of $3$. It is done in this paper. Our main result is the following theorem.
\begin{theorem}
\label{3n}
The shuffle group $G_{3,3n}$ is $2$-transitive if $3\mid n$ and $n$ is not a power of $3$.
\end{theorem}

The classification theorem for $G_{2,2n}$ can be seen in \cite{DGK1983}. For completeness, we summarize the classification theorem for $G_{k,kn}$ with $k\ge3$ as follows.
\begin{theorem}
\label{classification}
Let $G_{k,kn}$ be the shuffle group on $kn$ cards where $k\geq3$. Then the following statements hold.
\begin{enumerate}
  \item If $n=k^{e}$ ($e\ge0$), then $G_{k,kn}$ is the primitive wreath product of $\Sym([k])$ by the cyclic group of order $e+1$.
  \item If $k=4$ and $n=2^{2e-1}$ ($e\ge1$), then $G_{k,kn}$ is $(\mathbb{Z}_2)^{2e+1}\rtimes \mathrm{GL}(2e+1,2)$, the affine group of an $(2e+1)$-dimensional vector space over the field of order $2$.
  \item If $n\equiv2\pmod4$ and $k\equiv0~\mbox{or}~1\pmod4$, or $n\equiv0 \pmod4$ and $n$ is not a power of $k$, then $G_{k,kn}$ is the alternating group $\mathrm{Alt}([kn])$.
  \item In all other cases, $G_{k,kn}$ is the symmetric group $\Sym([kn])$.
\end{enumerate}
\end{theorem}

We will fix some notations in the next section and give the proof of Theorem \ref{3n} in Section 3.

\section{Preliminaries}
Throughout this section, let $n=3^st$ where $s$ and $t$ are integers satisfying $s>0$, $t>1$ and $3\nmid t$. For a nonnegative integer $m$ and a positive integer $\ell$, we use $[m]_{\ell}^{0}$ to denote the remainder of $\ell$ dividing $m$ and $[m]_{\ell}^{1}$ to denote the quotient of $\ell$ dividing $m-[m]_{\ell}^{0}$. Then $m=\ell[m]_{\ell}^{1}+[m]_{\ell}^{0}$. For every $x\in[3n]$ (note that $0\le x<3^{s+1}t$), we write $[x]_{t}^{1}$ in base $3$ as follows: $[x]_{t}^{1}=3^sx_s+\cdots+3x_1+x_0$ where $x_i\in[3]$ for every $i\in[s+1]$.
Therefore $x$ can be uniquely written as
\begin{equation*}
x=(3^sx_s+\cdots+3x_1+x_0)t+[x]_{t}^{0}.
\end{equation*}
For convenience, we identify $x$ with $(x_s,\ldots,x_1,x_0;X)$ where $X=[x]_{t}^{0}$. Note that sometimes we mix the two notations when doing addition. For example,
\begin{equation*}
(x_s,\ldots,x_3,0,1,1;t-1)+3^2t+2=(x_s,\ldots,x_3,1,1,2;1).
\end{equation*}

Now consider the shuffle group $G_{3,3n}=\langle\sigma,\rho_{\tau}\mid\tau\in \Sym([3])\rangle$ where $\sigma$ is the standard shuffle of a deck of $3n$ cards being evenly divided into $3$ piles. Recall that
\begin{equation*}
  (i+jn)^{\sigma}=3i+j~\mbox{and}~(i+jn)^{\rho_{\tau}}=i+j^{\tau}n
\end{equation*}
for all $i\in[n]$ and $j\in[3]$. Therefore
\begin{equation}\label{sigma}
  \begin{split}
     (x_s,\ldots,x_1,x_0;X)^{\sigma}=& (x_{s-1},\ldots,x_0,[3X+x_s]_{t}^{1};[3X+x_s]_{t}^{0}) \\
  = & (x_{s-1},\ldots,x_0,0;0)+3X+x_s,
  \end{split}
\end{equation}
\begin{equation}
\label{inverse}
(x_s,\ldots,x_1,x_0;X)^{\sigma^{-1}}
=([x_0t+X]_{3}^{0},x_{s},\ldots,x_2,x_1;[x_0t+X]_{3}^{1})
\end{equation}
and
\begin{equation}\label{ta}
(x_s,\ldots,x_1,x_0;X)^{\rho_{\tau}}
=(x_{s}^{\tau},\ldots,x_1,x_0;X)
\end{equation}
for every $(x_s,\ldots,x_1,x_0;X)\in[3n]$. By using Eq. (\ref{sigma})---(\ref{ta}), it is straightforward to calculate that
\begin{equation}
\label{spower}
(x_s,\ldots,x_1,x_0;X)^{\sigma^{i}}  =(x_{s-i},\ldots,x_{1},x_{0},0,\ldots,0;0)
+3^iX+\sum_{j=0}^{i-1}3^{i-1-j}x_{s-j}
\end{equation}
and
\begin{equation}
\label{tau}
(x_s,\ldots,x_1,x_0;X)^{\sigma^{i}\rho_{\tau}\sigma^{-i}}
=(x_s,\ldots,x_{s-i+1},x_{s-i}^{\tau},x_{s-i-1},
\ldots,x_1,x_0;X).
\end{equation}

\section{Proof of Theorem \ref{3n}}
Consider the subgroup $H:=\langle\sigma,\rho_{(01)}\rangle$ of $G_{3,3n}$ where $n=3^st$ with $s>0$, $t>1$ and $3\nmid t$. It is obvious that $H$ is contained in the stabilizer of $3n-1$. Our strategy for proving Theorem \ref{3n} is to prove that $H$ is transitive on $[3n-1]$. Let $x=(x_s,x_{s-1},\ldots,x_0;X)\in[3n]$. For every $i\in[s+1]$, set
\begin{equation*}
\bar{x}_i=\left\{
           \begin{array}{ll}
             1, & \hbox{if}~x_i=0; \\
             0, & \hbox{if}~x_i=1;  \\
             2, & \hbox{if}~x_i=2.
           \end{array}
         \right.
\end{equation*}
Write $\alpha_i=\sigma^{i}\rho_{(01)}\sigma^{-i}$ for every $i\in[s+1]$.
By Eq. (\ref{tau}) we have
\begin{equation*}
x^{\alpha_i}
=(x_s,\ldots,x_{s-\ell+1},\bar{x}_{s-i},x_{s-i-1},\ldots,x_0;X).
\end{equation*}
Set $\beta=\sigma^{-1}\rho_{(01)}\sigma$.
By using Eq. (\ref{sigma})---(\ref{ta}), it is straightforward to check that
\begin{equation*}
x^\beta=\left\{
           \begin{array}{ll}
             x+1, & \hbox{if}~[x_0t+X]_{3}^{0}=0; \\
             x-1, & \hbox{if}~[x_0t+X]_{3}^{0}=1; \\
             x, & \hbox{if}~[x_0t+X]_{3}^{0}=2.
           \end{array}
         \right.
\end{equation*}
The reader should keep in mind that we will use the above two formulas repeatedly without giving a refer in the proof of Theorem \ref{3n}.

\smallskip
Set $T(x)=|\{i\in[s+1]\mid x_i=2\}|$. The following lemma will be used in the proof of Theorem \ref{3n}.
\begin{lem}
\label{xy}
Let $x=(x_{s},x_{s-1},\ldots,x_1,x_0;X)\in [3n-1]$. If $1\le T(x)<s+1$, then there exists $y=(y_{s},y_{s-1},\ldots,y_1,y_0;Y)\in x^{H}$ such that either $T(y)=0$ or $y_0=0$, $y_1=2$ and $T(x)\geq T(y)$.
\end{lem}
\begin{proof}
Let $\ell$ be the minimum integer such that $x_\ell\ne2$. Set
\begin{equation*}
z=:x^{\sigma^{-\ell}}=(z_{s},z_{s-1},\ldots,z_1,z_0;Z).
\end{equation*}
By Eq. (\ref{inverse}), we have $z_{s-\ell}=x_s, z_{s-\ell-1}=x_{s-1},\ldots,z_{1}=x_{\ell+1},z_0=x_{\ell}$. Since $x_0=\cdots=x_{\ell-1}=2$, we have $T(x)\ge T(z)$. If $T(z)=0$, then we confirm the lemma by setting $y:=z$. In what follows, we assume $T(z)>0$.

Since $z_0=x_{\ell}\ne2$ and $3\nmid t$, either $[z_0t+Z]_{3}^{0}\ne 2$ or $[\bar{z}_0t+Z]_{3}^{0}\ne 2$. Set
\begin{equation*}
\mu_0=\left\{
  \begin{array}{ll}
    \sigma^{-1}, & \hbox{if}~[z_0t+Z]_{3}^{0}\ne2; \\
    \alpha_{s}\sigma^{-1}, & \hbox{if}~[z_0t+Z]_{3}^{0}=2.
  \end{array}
\right.
\end{equation*}
By Eq. (\ref{inverse}) and (\ref{tau}), we get $T(z^{\mu_0})=T(z)$. Similarly, if $z_1\ne2$, then there exists $\mu_1\in H$ such that $T(z^{\mu_0\mu_1})=T(z^{\mu_0})$.

\smallskip
Let $j$ be the minimum integer such that $z_j\ne2$ and $z_{j+1}=2$. Without loss of generality, we assume $z_j=0$ (as we can replace $z$ with $z^{\alpha_{s-j}}$ if $z_j=1$). Then we obtain the following result recursively: There exist $\mu_0,\ldots,\mu_{j-1}\in H$ such that
\begin{equation*}
T(z^{\mu_0\mu_1\cdots\mu_{j-1}})=\cdots=T(z^{\mu_0\mu_1})=T(z^{\mu_0})=T(z).
\end{equation*}
Set
\begin{equation*}
y=:z^{\mu_0\mu_1\cdots\mu_{j-1}}=(y_{s},y_{s-1},\ldots,y_1,y_0;Y).
\end{equation*}
Then $T(y)=T(z)$, $y_0=z_j=0$ and $y_1=z_{j+1}=2$. Since $z\in x^H$ and $T(x)\ge T(z)$, we have $y\in x^H$ and $T(x)\ge T(y)$.
\end{proof}

\medskip
Now we are ready to rove Theorem \ref{3n}.

\smallskip
\begin{proof}[Proof of Theorem \ref{3n}]
Let $n=3^st$ where $s>0$, $t>1$ and $3\nmid t$.
Let $H$ be the subgroup of $G_{3,3n}$ generated by $\sigma$ and $\rho_{(01)}$. Then $H$ is contained in the stabilizer of $3n-1$ in $G_{3,3n}$. Note that $G_{3,3n}$ is $2$-transitive if and only if $H$ is transitive on $[3n-1]$. It suffices to prove that $[3n-1]\subseteq 0^H$.  Let
\begin{equation*}
x=(x_{s},x_{s-1},\ldots,x_1,x_0;X)\in [3n-1]
\end{equation*}
where $X\in[t]$ and $x_i\in[3]$ for all $i\in[s+1]$. We will complete the proof in three steps as follows.

\medskip
\textsf{Step 1.}  Show that $[t]\subseteq0^{H}$.

\smallskip
Let $x\in[t]^{*}$. Then $x=X$ and $x_i=0$ for every $i\in[s+1]$.
If $[x]_{3}^{0}=0$, then
 $x^{\sigma^{-1}} = \frac{x}{3}$.
If $[x]_{3}^{0}=1$, then $x^{\beta}=x-1$.
If $[x]_{3}^{0}=2$ and $[t]_{3}^{0}=2$ , then $[t+x]_{3}^{0}=1$  and therefore
\begin{eqnarray*}
 x^{\alpha_{s}\beta\alpha_{s}}
=(t+x)^{\beta\alpha_{s}}=(t+x-1)^{\alpha_{s}}= x-1.
\end{eqnarray*}
If $[x]_{3}^{0}=2$  and $[t]_{3}^{0}=1$ , then $[t+x]_{3}^{0}=0$  and therefore
\begin{eqnarray*}
 x^{\alpha_{s}\beta\alpha_{s}\sigma^{-1}}= (t+x)^{\beta\alpha_{s}\sigma^{-1}}=(t+x+1)^{\alpha_{s}\sigma^{-1}}
=(x+1)^{\sigma^{-1}}=\frac{x+1}{3}.
\end{eqnarray*}
Note that we have proved that $x^H$ always contains an integer less than $x$ for all  $x\in[t]^{*}$. Therefore $[t]\subseteq0^{H}$.

\medskip
\textsf{Step 2.}  Show that $x\in0^{H}$ if $x\notin[t]$ and
$x_i\ne2$ for all $i\in[s+1]$.

\smallskip

Assume that $x\notin[t]$ and $x_i\ne2$ for all $i\in[s+1]$.
Then there exists a nonempty subset $I$ of $[s+1]$ such that $x_i=1$ if $i\in I$ and $x_i=0$ if $i\in [s+1]\setminus I$. Since $[t]\subseteq0^{H}$ and
$x^{\prod_{i\in I}\alpha_{s-i}}=X\in[t]$,
we have $x\in0^{H}$.

\medskip
\textsf{Step 3.}  Show that $x\in 0^{H}$ for all $x\in[3n-1]$.

\smallskip

Recall that $T(x)=|\{i\in[s+1]\mid x_i=2\}|$. We proceed by induction on $T(x)$ to show that $x\in 0^{H}$ for all $x\in[3n-1]$. It has been proved that $x\in 0^{H}$ if $T(x)=0$ in the the first two steps. Now let $T(x)\geq1$ and suppose that $y\in 0^{H}$ for all $y\in[3n-1]$ satisfying $T(y)<T(x)$.
We will complete the proof by verifying that there exists $b\in x^H$ such that $T(b)<T(x)$.

\smallskip
If $T(x)=s+1$, then $x_i=2$ for every $i\in[s+1]$. Therefore
\begin{equation*}
x=2(3^s+\cdots+3+1)t+X=3^{s+1}t-t+X=3n-(t-X).
\end{equation*}
By Eq. (\ref{spower}), we have
\begin{align*}
 x^{\sigma^{s+1}}& =3^{s+1}X+3^{s}x_s+3^{s-1}x_{s-1}+
   \cdots+3x_{1}+x_{0} \\
 & =3^{s+1}X+2(3^{s}+3^{s-1}+
   \cdots+3+1) \\
 & =3^{s+1}(X+1)-1.
\end{align*}
Since $x=3n-(t-X)<3n-1$, we have $X<t-1$ and it follows that
\begin{equation*}
x-x^{\sigma^{s+1}}=3^{s+1}t-t+X-3^{s+1}(X+1)+1=
(3^{s+1}-1)(t-X-1)>0.
\end{equation*}
Thus $x>x^{\sigma^{s+1}}$.
Similarly, if $T(x^{\sigma^{i(s+1)}})=s+1$ for some positive integer $i$,
then $x^{\sigma^{i(s+1)}}>x^{\sigma^{(i+1)(s +1)}}$. Since $x$ is finite, there exists a positive integer $\ell$ such that $T(x^{\sigma^{\ell(s+1)}})<s+1$. Set $b=x^{\sigma^{\ell(s+1)}}$. Then $b\in x^H$ and $T(b)<T(x)$.

\smallskip
Now we assume that $1\le T(x)<s+1$.  By Lemma \ref{xy}, we can further assume $x_0=0$ and $x_1=2$. Then $x=(x_s,x_{s-1},\ldots,x_2,2,0;X)$. Since $3\nmid t$, we have $[t]_{3}^{0}=1~\mbox{or}~2$. So we can proceed with the proof in two cases.

\medskip
\textsf{Case 1.} $[t]_{3}^{0}=2$.

\smallskip
First we have the following observations.
\begin{enumerate}
  \item[(a)] If $[X]_{3}^{0}=0$, then
 \begin{equation*}
  x=(x_s,x_{s-1},\ldots,x_2,2,0;X)\equiv0\pmod3,
 \end{equation*}
and
\begin{align*}
  x^{\beta\alpha_s}= &(x_s,x_{s-1},\ldots,x_2,2,0;X)^{\beta\alpha_s}\\
   = &(x_s,x_{s-1},\ldots,x_2,2,0;X+1)^{\alpha_s}\\
   = &(x_s,x_{s-1},\ldots,x_2,2,1;X+1)\equiv0\pmod3.
\end{align*}
  \item[(b)] If $[X]_{3}^{0}=1$, then
\begin{align*}
  x^{\alpha_s}= &(x_s,x_{s-1},\ldots,x_2,2,1;X)\equiv0\pmod3
\end{align*}
and
\begin{align*}
  x^{\beta}
 = &(x_s,x_{s-1},\ldots,x_2,2,0;X-1)\equiv0\pmod3.
\end{align*}
  \item[(c)] If $[X]_{3}^{0}=2$, then $[t+X]_{3}^{0}=1$ and therefore
\begin{align*}
  x^{\alpha_s\beta}
 = &(x_s,x_{s-1},\ldots,x_2,2,0;X)^{\alpha_s\beta} \\
 = &(x_s,x_{s-1},\ldots,x_2,2,1;X)^{\beta} \\
  = & (x_s,x_{s-1},\ldots,x_2,2,1;X-1)\equiv0\pmod3
\end{align*}
and
\begin{align*}
  x^{\alpha_s\beta\alpha_s\beta}
  = & (x_s,x_{s-1},\ldots,x_2,2,1;X-1)^{\alpha_s\beta}\\
  = & (x_s,x_{s-1},\ldots,x_2,2,0;X-1)^{\beta}\\
  = & (x_s,x_{s-1},\ldots,x_2,2,0;X-2)\equiv0\pmod3.
\end{align*}
\end{enumerate}
Set
\begin{equation*}
  y=(0,x_s,x_{s-1},\ldots,x_2,2;Y)
\end{equation*}
and
\begin{equation*}
  z=(0,x_s,x_{s-1},\ldots,x_2,2;Z)
\end{equation*}
where
\begin{equation*}
Y=\left\{
  \begin{array}{ll}
    X/3, & \hbox{if}~[X]_{3}^{0}=0; \\
    (X-1)/3, & \hbox{if}~[X]_{3}^{0}=1; \\
    (X-2)/3, & \hbox{if}~[X]_{3}^{0}=2 \\
  \end{array}
\right.
\end{equation*}
and
\begin{equation*}
Z=\left\{
  \begin{array}{ll}
    (t+X+1)/3, & \hbox{if}~[X]_{3}^{0}=0; \\
    (t+X)/3, & \hbox{if}~[X]_{3}^{0}=1; \\
    (t+X-1)/3, & \hbox{if}~[X]_{3}^{0}=2. \\
  \end{array}
\right.
\end{equation*}
By observations (a)--(c), we have $y,z\in x^{H}$.

\smallskip
If $[2t+Y]_{3}^{0}\ne2$, then we set $b=y^{\sigma^{-1}}$. If $[2t+Y]_{3}^{0}=2$ but $[2t+Z]_{3}^{0}\ne2$ then we set $b=z^{\sigma^{-1}}$. In either cases, we have $b\in x^H$ and $T(b)=T(x)-1$.

\smallskip
Now we assume $[2t+Y]_{3}^{0}=[2t+Z]_{3}^{0}=2$. Since $(t+1)/3=Z-Y$, it follows that $[(t+1)/3]_{3}^{0}=[Z-Y]_{3}^{0}=0$.
Set
\begin{equation*}
  w=(x_s,x_{s-1},\ldots,x_2,2,1;3Y).
\end{equation*}
Then
\begin{equation*}
  w^{\sigma^{-1}}=(2,x_s,x_{s-1},\ldots,x_2,2;(t-2)/3+Y).
\end{equation*}
Since $w=y^{\sigma\alpha_s}$, we get $w\in x^H$.
Since $[2t+Y]_{3}^{0}=2$ and $[(t+1)/3]_{3}^{0}=0$, we have
\begin{equation*}
2t+(t-2)/3+Y=2t+Y+(t+1)/3-1\equiv1\pmod3
\end{equation*}
 and it follows that
\begin{equation*}
  w^{\sigma^{-1}\beta}=(2,x_s,x_{s-1},\ldots,x_2,2;(t-2)/3+Y-1).
\end{equation*}
Thus
\begin{eqnarray*}
  w^{\sigma^{-1}\beta\sigma\alpha_{s}\sigma^{-1}} &=& (2,x_s,x_{s-1},\ldots,x_2,2;(t-2)/3+Y-1)^{\sigma\alpha_{s}\sigma^{-1}} \\
 &=& (x_s,x_{s-1},\ldots,x_2,2,1;3Y-3)^{\alpha_{s}\sigma^{-1}} \\
 &=& (x_s,x_{s-1},\ldots,x_2,2,0;3Y-3)^{\sigma^{-1}} \\
 &=& (0,x_s,x_{s-1},\ldots,x_2,2;Y-1).
\end{eqnarray*}
Set $b=w^{\sigma^{-1}\beta\sigma\alpha_{s}\sigma^{-2}}$. Since $[2t+Y]_{3}^{0}=2$, we have $[2t+Y-1]_{3}^{0}=1$ and therefore
\begin{eqnarray*}
  b &=& (0,x_s,x_{s-1},\ldots,x_2,2;Y-1)^{\sigma^{-1}} \\
   &=& (1,0,x_s,x_{s-1},\ldots,x_2;(2t+Y-2)/3).
\end{eqnarray*}
Thus $b\in x^H$ and $T(b)=T(x)-1$.

\medskip
\textsf{Case 2.} $t\equiv1\pmod3$.

\smallskip
Set
\begin{equation*}
  y=(0,x_s,x_{s-1},\ldots,x_2,2;Y)
\end{equation*}
and
\begin{equation*}
  z=(0,x_s,x_{s-1},\ldots,x_2,2;Z)
\end{equation*}
where
\begin{equation*}
Y=\left\{
  \begin{array}{ll}
    X/3, & \hbox{if}~X\equiv0\pmod3; \\
    (X-1)/3, & \hbox{if}~X\equiv1\pmod3; \\
    (X+1)/3, & \hbox{if}~X\equiv2\pmod3 \\
  \end{array}
\right.
\end{equation*}
and
\begin{equation*}
Z=\left\{
  \begin{array}{ll}
    (t+X-1)/3, & \hbox{if}~X\equiv0\pmod3; \\
    (t+X-2)/3, & \hbox{if}~X\equiv1\pmod3; \\
    (t+X)/3, & \hbox{if}~X\equiv2\pmod3. \\
  \end{array}
\right.
\end{equation*}
Similar to the discussion in \textsf{Case 1.}, we have $y,z\in x^{H}$.

\smallskip
If $[2t+Y]_{3}^{0}\ne2$, then we set $b=y^{\sigma^{-1}}$.
If $[2t+Y]_{3}^{0}=2$ but $[2t+Z]_{3}^{0}\ne2$ then we set $b=z^{\sigma^{-1}}$. In either cases, we have $b\in x^H$ and $T(b)=T(x)-1$.

\smallskip
Now we assume $[2t+Y]_{3}^{0}=[2t+Z]_{3}^{0}=2$. Then  $[(t-1)/3]_{3}^{0}=[Z-Y]_{3}^{0}=0$.
Set
\begin{equation*}
  w=(x_s,x_{s-1},\ldots,x_2,2,0;3Y-1).
\end{equation*}
Then
\begin{equation*}
  w^{\sigma^{-1}}=(2,x_s,x_{s-1},\ldots,x_2,2;Y-1).
\end{equation*}
Since $w=y^{\sigma\alpha_{s}\beta\alpha_{s}}$, we have $w\in x^H$.
Since $[2t+Y]_{3}^{0}=2$, we have $[2t+Y-1]_{3}^{0}=1$ and it follows that
\begin{equation*}
  w^{\sigma^{-1}\beta}=(2,x_s,x_{s-1},\ldots,x_2,2;Y-2).
\end{equation*}
Thus
\begin{eqnarray*}
  w^{\sigma^{-1}\beta\sigma\alpha_s\sigma^{-1}} &=& (2,x_s,x_{s-1},\ldots,x_2,2;Y-2)^{\sigma\alpha_s\sigma^{-1}} \\
 &=& (x_s,x_{s-1},\ldots,x_2,2,0;3Y-4)^{\alpha_s\sigma^{-1}} \\
 &=& (x_s,x_{s-1},\ldots,x_2,2,1;3Y-4)^{\sigma^{-1}} \\
 &=& (0,x_s,x_{s-1},\ldots,x_2,2;(t-1)/3+Y-1)\\
 &=& (0,x_s,x_{s-1},\ldots,x_2,2;Z-1).
\end{eqnarray*}
Set $b=w^{\sigma^{-1}\beta\sigma\alpha_s\sigma^{-2}}$.
Since $[2t+Z]_{3}^{0}=2$, we have $[2t+Z-1]_{3}^{0}=1$ and therefore
\begin{eqnarray*}
  b &=& (0,x_s,x_{s-1},\ldots,x_2,2;Z-1)^{\sigma^{-1}} \\
   &=& (1,0,x_s,x_{s-1},\ldots,x_2;(2t+Z-2)/3).
\end{eqnarray*}
Thus $b\in x^H$ and $T(b)=T(x)-1$.

\medskip
Now we have proved that there always exists $b\in x^H$ such that $T(x)> T(b)$. By induction hypothesis, we get $b\in 0^H$. It follows that $x\in 0^H$. Therefore $H$ is transitive on $[3n-1]$.
\end{proof}

\end{document}